\documentclass{article}
\usepackage{amssymb}
\usepackage{amsmath}
\usepackage{amsfonts}

\newtheorem{theorem}{Theorem}

\newtheorem{lemma}[theorem]{Lemma}

\newenvironment{proof}[1][Proof]{\noindent\textbf{#1} }{\ \rule{0.5em}{0.5em}}
\input{tcilatex}
\begin{document}

\begin{center}
\vspace{1.5cm}

\bigskip

{\Large \textsf{A solution to Open Problem 3.137 by O. Furdui }}\\[0pt]
{\Large \textsf{on multiple factorial series}}

\vspace{1.5cm}

{\large \textbf{Ulrich Abel}}\medskip

{\large \textit{Technische Hochschule Mittelhessen}}\\[0pt]
{\large \textit{Fachbereich MND\\[0pt]
Wilhelm-Leuschner-Stra\ss e 13, 61169 Friedberg, Germany}}\\[0pt]
{\large \textit{(E-mail: Ulrich.Abel@\textit{mnd.thm.de})}}\\[0pt]
\vspace{1.5cm}

\vspace{2cm}
\end{center}

\vfill


\vfill
{\large \textbf{Abstract.}}

\bigskip

In this paper we give a closed expression for the series\ 
\begin{equation*}
\dsum\limits_{n_{1}=1}^{\infty }\cdots \dsum\limits_{n_{k}=1}^{\infty }\frac{%
n_{1}\cdots n_{k}}{\left( n_{1}+\cdots +n_{k}\right) !},
\end{equation*}%
for all $k=1,2,3,\ldots $, solving Open Problem 3.137 in the recent book 
\cite[Chapt. 3.7, problem 3.137]{Furdui-book-2013} by Furdui. The method is
based on properties of divided differences. It applies also to similar
series and certain generalizations.

\bigskip



\vfill
\newpage

\makeatletter%
\renewcommand{\@evenhead}{\thepage\hfill{\small\sc
U. Abel}\hfill}%
\renewcommand{\@oddhead}{\hfill{\small\sc
A solution to Open Problem 3.137 by O. Furdui}\hfill\thepage} %

\section{Introduction}

In his recent book \cite[Chapt. 3.7, Problem 3.137]{Furdui-book-2013} O.
Furdui states the open problem to give closed expressions for the multiple
factorial series 
\begin{equation*}
S_{k}:=\dsum\limits_{n_{1}=1}^{\infty }\cdots \dsum\limits_{n_{k}=1}^{\infty
}\frac{n_{1}\cdots n_{k}}{\left( n_{1}+\cdots +n_{k}\right) !},
\end{equation*}%
for all integers $k\geq 4$. Moreover, he conjectured that $S_{k}$ is, for
all integers $k\in \mathbb{N}$, a rational multiple of Euler's number $e$,
i.e., $S_{k}=a_{k}e$ with $a_{k}\in \mathbb{Q}$. It is easy to see that $%
S_{1}=e$. Using the Beta function technique Furdui \cite[Problem 3.114 and
3.118, respecively]{Furdui-book-2013} shows $a_{2}=2/3$ and $a_{3}=31/120$.

More generally, Furdui considers the series 
\begin{eqnarray*}
S_{k,0} &:&=\dsum\limits_{n_{1}=1}^{\infty }\cdots
\dsum\limits_{n_{k}=1}^{\infty }\frac{1}{\left( n_{1}+\cdots +n_{k}\right) !}%
, \\
S_{k,j} &:&=\dsum\limits_{n_{1}=1}^{\infty }\cdots
\dsum\limits_{n_{k}=1}^{\infty }\frac{n_{1}\cdots n_{j}}{\left( n_{1}+\cdots
+n_{k}\right) !}\text{ \qquad }\left( 1\leq j\leq k\right) .
\end{eqnarray*}%
Obviously, we have $S_{k}=S_{k,k}$. Furdui determines the exact values $%
S_{k,1}=\left( k!\right) ^{-1}e$ and $S_{3,2}=\left( 5/24\right) e$ \cite[%
Problems 3.117 and 3.120, respectively]{Furdui-book-2013}. Also an
expression for $S_{k,0}$ is given \cite[Problem 3.119]{Furdui-book-2013}: 
\begin{equation}
S_{k,0}=\left( -1\right) ^{k}\left( 1-e\dsum\limits_{j=0}^{k-1}\frac{\left(
-1\right) ^{j}}{j!}\right)  \label{expression-for-Sk0}
\end{equation}%
More generally, one defines, for real numbers $x_{1},\ldots ,x_{k}$, the
function 
\begin{equation}
S_{k}\left( x_{1},\ldots ,x_{k}\right) :=\dsum\limits_{n_{1}=1}^{\infty
}\cdots \dsum\limits_{n_{k}=1}^{\infty }\frac{x_{1}^{n_{1}}\cdots
x_{k}^{n_{k}}}{\left( n_{1}+\cdots +n_{k}\right) !}.  \label{def-sk(x)}
\end{equation}%
Closed expressions for $S_{k}\left( x_{1},\ldots ,x_{k}\right) $ in the
special case $k=2$ can be found in \cite[Problem 3.115 (see also Problem
3.116)]{Furdui-book-2013}.

In this note we give an affirmative answer on Furdui's conjecture $%
e^{-1}S_{k}=a_{k}\in \mathbb{Q}$ and provide an explicit representation of $%
a_{k}$ in the form 
\begin{equation*}
a_{k}=\frac{1}{\left( 2k-1\right) !}\left. \left[ \left( \frac{d}{dx}\right)
^{2k-1}\left( x^{k-1}e^{x}\right) \right] \right\vert _{x=1}.
\end{equation*}%
Moreover, we derive similar expressions for $S_{k,j}$. Our main result
considers even more general sums. Finally, we represent $S_{k}\left(
x_{1},\ldots ,x_{k}\right) $ as a finite sum, for all $k\in \mathbb{N}$.

The proofs are based on divided differences. For pairwise different real or
complex numbers $x_{0},\ldots ,x_{k}$, in most textbooks, the divided
differences of a function $f$ are defined recursively:\ $\left[ x_{0};f%
\right] =f\left( x_{0}\right) $, \ldots , 
\begin{equation*}
\left[ x_{0},\ldots ,x_{k};f\right] =\frac{\left[ x_{1},\ldots ,x_{k};f%
\right] -\left[ x_{0},\ldots ,x_{k-1};f\right] }{x_{k}-x_{0}}.
\end{equation*}

\section{Main results}

Let 
\begin{equation*}
g\left( z\right) =\dsum\limits_{n=0}^{\infty }g_{n}z^{n}
\end{equation*}%
be a power series converging for $\left\vert z\right\vert <R$ with $R>1$.
For integers $\ell \geq 0$, put 
\begin{equation*}
g_{\ell }\left( z\right) =\dsum\limits_{n=0}^{\infty }g_{n+\ell }z^{n}.
\end{equation*}%
Hence $g_{0}=g$ and, for $\ell \geq 1$, 
\begin{equation*}
z^{\ell }g_{\ell }\left( z\right) =g\left( z\right)
-\dsum\limits_{n=0}^{\ell -1}g_{n}z^{n}.
\end{equation*}%
For $k\in \mathbb{N}$, define 
\begin{equation}
G_{k,\ell }\left( x_{1},\ldots ,x_{k}\right) =\dsum\limits_{n_{1}=0}^{\infty
}\cdots \dsum\limits_{n_{k}=0}^{\infty }g_{n_{1}+\cdots +n_{k}+\ell }\cdot
x_{1}^{n_{1}}\cdots x_{k}^{n_{k}}.  \label{def-Gkl}
\end{equation}

Our main result are presented in the following theorems.

\begin{theorem}
\label{th-main-theorem-1}With the above notation, for all $k\in \mathbb{N}$
and integers $\ell \geq 0$,$\ $%
\begin{equation*}
G_{k,\ell }\left( x_{1},\ldots ,x_{k}\right) =\left[ x_{1},\ldots
,x_{k};z^{\ell -1}g_{\ell }\left( z\right) \right] .
\end{equation*}
\end{theorem}

\begin{theorem}
\label{th-main-theorem-2}Let $k,j$ be integers such that $1\leq j\leq k$ and
let $i_{1},\ldots ,i_{j}\in \left\{ 1,\ldots ,k\right\} $ be pairwise
different integers. Then, 
\begin{equation*}
\lim_{x_{1},\ldots ,x_{k}\rightarrow x}\frac{\partial ^{j}}{\partial
x_{i_{1}}\cdots \partial x_{i_{j}}}G_{k,\ell }\left( x_{1},\ldots
,x_{k}\right) =\frac{1}{\left( k+j-1\right) !}\left. \left[ \left( \frac{d}{%
dz}\right) ^{k+j-1}z^{k-1}g_{\ell }\left( z\right) \right] \right\vert
_{z=x}.
\end{equation*}
\end{theorem}

For convenience, we define, for $k,\ell \in \mathbb{N}$ and real numbers $%
x_{1},\ldots ,x_{k}$, 
\begin{equation}
f_{k,\ell }\left( x_{1},\ldots ,x_{k}\right)
:=\dsum\limits_{n_{1}=0}^{\infty }\cdots \dsum\limits_{n_{k}=0}^{\infty }%
\frac{x_{1}^{n_{1}}\cdots x_{k}^{n_{k}}}{\left( n_{1}+\cdots +n_{k}+\ell
\right) !}  \label{def-fkl(x)}
\end{equation}%
In the special case of the exponential function $g=\exp $, Theorem~\ref%
{th-main-theorem-1} provides the representation 
\begin{equation}
f_{k,\ell }\left( x_{1},\ldots ,x_{k}\right) =\left[ x_{1},\ldots
,x_{k};x^{\ell -1}\exp _{\ell }\left( x\right) \right] .
\label{fkl-as-divdiff}
\end{equation}%
With regard to the series $S_{k,j}$ as defined in the Introduction it
follows that 
\begin{eqnarray*}
S_{k,j} &=&\dsum\limits_{n_{1}=0}^{\infty }\cdots
\dsum\limits_{n_{j}=0}^{\infty }\dsum\limits_{n_{j+1}=1}^{\infty }\cdots
\dsum\limits_{n_{k}=1}^{\infty }\frac{n_{1}\cdots n_{j}}{\left( n_{1}+\cdots
+n_{k}\right) !} \\
&=&\dsum\limits_{n_{1}=0}^{\infty }\cdots \dsum\limits_{n_{k}=0}^{\infty }%
\frac{n_{1}\cdots n_{j}}{\left( n_{1}+\cdots +n_{k}+k-j\right) !} \\
&=&\frac{\partial ^{j}f_{k,k-j}}{\partial x_{1}\cdots \partial x_{j}}\left(
1,\ldots ,1\right) .
\end{eqnarray*}%
Hence, Theorem~\ref{th-main-theorem-1} implies the following theorem as an
immediate corollary.

\begin{theorem}
\label{th-expression-Skj}Let $k,j$ be integers such that $0\leq j\leq k$.
Then the series $S_{k,j}$ possess the representations 
\begin{equation*}
S_{k,j}=\frac{1}{\left( k+j-1\right) !}\left. \left[ \left( \frac{d}{dz}%
\right) ^{k+j-1}z^{k-1}\exp _{k-j}\left( z\right) \right] \right\vert _{z=1}.
\end{equation*}
\end{theorem}

In the special case $j=0$, we obtain 
\begin{equation*}
S_{k,0}\equiv \dsum\limits_{n_{1}=1}^{\infty }\cdots
\dsum\limits_{n_{k}=1}^{\infty }\frac{1}{\left( n_{1}+\cdots +n_{k}\right) !}%
=\frac{1}{\left( k-1\right) !}\left. \left[ \left( \frac{d}{dz}\right) ^{k-1}%
\frac{e^{z}-1}{z}\right] \right\vert _{z=1}
\end{equation*}%
and an application of the Leibniz rule immediately leads to formula $\left( %
\ref{expression-for-Sk0}\right) $. In the cases $1\leq j\leq k$ the formula
of Theorem~\ref{th-expression-Skj} simplifies to 
\begin{equation*}
S_{k,j}=\frac{1}{\left( k+j-1\right) !}\left. \left[ \left( \frac{d}{dz}%
\right) ^{k+j-1}z^{j-1}e^{z}\right] \right\vert _{z=1}.
\end{equation*}%
Application of the Leibniz rule yields the explicit formula 
\begin{equation*}
S_{k,j}=e\sum_{i=0}^{j-1}\binom{j-1}{i}\frac{1}{\left( k+j-i-1\right) !}.
\end{equation*}%
Hence, the series $S_{k,j}$ are rational multiples of $e$ for $j=1,\ldots ,k$%
. We list some initial values:\ 
\begin{equation*}
\begin{tabular}{c|cccccc}
$k\backslash j$ & $0$ & $1$ & $2$ & $3$ & $4$ & $5$ \\ \hline
\multicolumn{1}{r|}{$1$} & \multicolumn{1}{|r}{$e-1$} & \multicolumn{1}{r}{$%
1 $} & \multicolumn{1}{r}{} & \multicolumn{1}{r}{} & \multicolumn{1}{r}{} & 
\multicolumn{1}{r}{} \\ 
\multicolumn{1}{r|}{$2$} & \multicolumn{1}{|r}{$1$} & \multicolumn{1}{r}{$%
1/2 $} & \multicolumn{1}{r}{$2/3$} & \multicolumn{1}{r}{} & 
\multicolumn{1}{r}{} & \multicolumn{1}{r}{} \\ 
\multicolumn{1}{r|}{$3$} & \multicolumn{1}{|r}{$e/2-1$} & \multicolumn{1}{r}{%
$1/6$} & \multicolumn{1}{r}{$5/24$} & \multicolumn{1}{r}{$31/120$} & 
\multicolumn{1}{r}{} & \multicolumn{1}{r}{} \\ 
\multicolumn{1}{r|}{$4$} & \multicolumn{1}{|r}{$1-e/3$} & \multicolumn{1}{r}{%
$1/24$} & \multicolumn{1}{r}{$1/20$} & \multicolumn{1}{r}{$43/720$} & 
\multicolumn{1}{r}{$179/2520$} & \multicolumn{1}{r}{} \\ 
\multicolumn{1}{r|}{$5$} & \multicolumn{1}{|r}{$3e/8-1$} & 
\multicolumn{1}{r}{$1/120$} & \multicolumn{1}{r}{$7/720$} & 
\multicolumn{1}{r}{$19/1680$} & \multicolumn{1}{r}{$529/40320$} & 
\multicolumn{1}{r}{$787/51840$}%
\end{tabular}%
\end{equation*}%
We close with the special case $j=k$: 
\begin{equation*}
S_{k}\equiv S_{k,k}=e\sum_{i=0}^{k-1}\binom{k-1}{i}\frac{1}{\left(
2k-i-1\right) !}.
\end{equation*}%
For convenience of the reader we list some exact and numerical values of $%
a_{k}=e^{-1}S_{k}$: 
\begin{equation*}
\begin{tabular}{c|ll}
$k$ & $a_{k}$ &  \\ \hline
\multicolumn{1}{r|}{$1$} & \multicolumn{1}{|r}{$1$} & $=1.000000$ \\ 
\multicolumn{1}{r|}{$2$} & \multicolumn{1}{|r}{$2/3$} & $\approx 0.666667$
\\ 
\multicolumn{1}{r|}{$3$} & \multicolumn{1}{|r}{$31/120$} & $\approx 0.258333$
\\ 
\multicolumn{1}{r|}{$4$} & \multicolumn{1}{|r}{$179/2520$} & $\approx
0.0710317$ \\ 
\multicolumn{1}{r|}{$5$} & \multicolumn{1}{|r}{$787/51840$} & $\approx
0.0151813$ \\ 
\multicolumn{1}{r|}{$10$} & \multicolumn{1}{|r}{} & $5.912338752837942\cdot
10^{-7}$ \\ 
\multicolumn{1}{r|}{$100$} & \multicolumn{1}{|r}{} & $2.829019570367539\cdot
10^{-158}$%
\end{tabular}%
\end{equation*}%
\bigskip

Finally, we mention that the series $S_{k}\left( x_{1},\ldots ,x_{k}\right) $
as defined in $\left( \ref{def-sk(x)}\right) $ is connected to the function $%
f_{k,\ell }$ as defined in $\left( \ref{def-fkl(x)}\right) $\ by the
relation 
\begin{equation*}
S_{k}\left( x_{1},\ldots ,x_{k}\right) =x_{1}\cdots x_{k}\cdot f_{k,k}\left(
x_{1},\ldots ,x_{k}\right) .
\end{equation*}%
Hence, by Eq. $\left( \ref{fkl-as-divdiff}\right) $, we have the new
approach 
\begin{equation*}
S_{k}\left( x_{1},\ldots ,x_{k}\right) =x_{1}\cdots x_{k}\cdot \left[
x_{1},\ldots ,x_{k};x^{k-1}\exp _{k}\left( x\right) \right] .
\end{equation*}%
Experiments with different functions $g$ may be subject of further studies.

\section{Auxiliary results and proofs}

Let $x_{0},\ldots ,x_{k}$ be pairwise different real or complex numbers. In
most textbooks, the divided differences of a function $f$ are defined
recursively:\ $\left[ x_{0};f\right] =f\left( x_{0}\right) $, \ldots , 
\begin{equation*}
\left[ x_{0},\ldots ,x_{k};f\right] =\frac{\left[ x_{1},\ldots ,x_{k};f%
\right] -\left[ x_{0},\ldots ,x_{k-1};f\right] }{x_{k}-x_{0}}
\end{equation*}%
In this paper we make use of the some properties of divided differences
gathered in the following lemmas.

\begin{lemma}
\label{lemma-divdiff-representation}The divided differences possess the
integral representation 
\begin{equation*}
\left[ x_{0},\ldots ,x_{k};f\right] =\int_{0}^{1}\int_{0}^{t_{1}}\cdots
\int_{0}^{t_{k-1}}f^{\left( k\right) }\left( x_{0}+\left( x_{1}-x_{0}\right)
t_{1}+\cdots +\left( x_{k}-x_{k-1}\right) t_{k}\right) dt_{k}\cdots
dt_{2}dt_{1},
\end{equation*}%
provided that $f^{\left( k-1\right) }$ is absolutely continuous.
\end{lemma}

This can be proved by induction on $k$ (see, e.g., \cite[Chapt. 4, \S 7, Eq.
(7.12) and below]{DeVore-Lorentz-Book-1993}).

\begin{lemma}
\label{lemma-divdiff-derivatives}Let $1\leq j\leq k$ and let $i_{1},\ldots
,i_{j}\in \left\{ 1,\ldots ,k\right\} $ be pairwise different integers.
Then, for each function $f$ having a derivative of order $k+j-1$, 
\begin{equation*}
\lim_{x_{1},\ldots ,x_{k}\rightarrow x}\frac{\partial ^{j}\left[
x_{1},\ldots ,x_{k};f\right] }{\partial x_{i_{1}}\cdots \partial x_{i_{j}}}=%
\frac{1}{\left( k+j-1\right) !}f^{\left( k+j-1\right) }\left( x\right)
\end{equation*}
\end{lemma}

\begin{proof}[Proof of Lemma~\protect\ref{lemma-divdiff-representation}]
Because the divided differences are invariant with respect to the order of
knots we can restrict ourselves to the case $i_{\nu }=\nu $ $\left( \nu
=1,\ldots ,j\right) $. By Lemma~\ref{lemma-divdiff-representation}, we have 
\begin{eqnarray*}
&&\frac{\partial ^{j}\left[ x_{1},\ldots ,x_{k};f\right] }{\partial
x_{1}\cdots \partial x_{j}} \\
&=&\frac{\partial ^{j}}{\partial x_{1}\cdots \partial x_{j}}%
\int_{0}^{1}\int_{0}^{t_{1}}\cdots \int_{0}^{t_{k-2}}f^{\left( k-1\right)
}\left( x_{1}+\left( x_{2}-x_{1}\right) t_{1}+\cdots +\left(
x_{k}-x_{k-1}\right) t_{k-1}\right) dt_{k-1}\cdots dt_{2}dt_{1} \\
&=&\int_{0}^{1}\int_{0}^{t_{1}}\cdots \int_{0}^{t_{k-2}}f^{\left(
k+j-1\right) }\left( x_{1}\left( 1-t_{1}\right) +x_{2}\left(
t_{1}-t_{2}\right) +\cdots +x_{k}\left( t_{k-1}-t_{k}\right) \right) \\
&&\times \left( 1-t_{1}\right) \left( t_{1}-t_{2}\right) \cdots \left(
t_{j-1}-t_{j}\right) dt_{k-1}\cdots dt_{2}dt_{1},
\end{eqnarray*}%
where we put $t_{k}=0$. Taking the limit we obtain 
\begin{eqnarray*}
&&\lim_{x_{1},\ldots ,x_{k}\rightarrow x}\frac{\partial ^{j}\left[
x_{1},\ldots ,x_{k};f\right] }{\partial x_{1}\cdots \partial x_{j}} \\
&=&f^{\left( k+j-1\right) }\left( x\right)
\int_{0}^{1}\int_{0}^{t_{1}}\cdots \int_{0}^{t_{k-2}}\left( 1-t_{1}\right)
\left( t_{1}-t_{2}\right) \cdots \left( t_{j-1}-t_{j}\right) dt_{k-1}\cdots
dt_{2}dt_{1}.
\end{eqnarray*}%
An inductive argument shows that the multiple integral has the value $%
1/\left( k+j-1\right) !$ which completes the proof of Lemma~\ref%
{lemma-divdiff-derivatives}.
\end{proof}

Popoviciu \cite{Popoviciu-1940} proved the following formula for monomials.

\begin{lemma}
\label{lemma-popoviciu}For each integer $r\geq 0$, 
\begin{equation*}
\left[ x_{0},\ldots ,x_{k};z^{k+r}\right] =\sum x_{0}^{n_{0}}\cdots
x_{k}^{n_{k}},
\end{equation*}%
where the sum runs over all nonnegative integers $n_{0},\ldots ,n_{k}$
satisfying $n_{0}+\cdots +n_{k}=r$.
\end{lemma}

\section{Proof of the main theorems}

\begin{proof}[Proof of Theorem~\protect\ref{th-main-theorem-1}]
By Eq. $\left( \ref{def-Gkl}\right) $ and Lemma~\ref{lemma-popoviciu}, we
have 
\begin{eqnarray*}
G_{k,\ell }\left( x_{1},\ldots ,x_{k}\right) &=&\dsum\limits_{n=0}^{\infty
}g_{n+\ell }\dsum\limits_{n_{1}+\cdots +n_{k}=n}x_{1}^{n_{1}}\cdots
x_{k}^{n_{k}}=\dsum\limits_{n=0}^{\infty }g_{n+\ell }\left[ x_{1},\ldots
,x_{k};z^{k-1+n}\right] \\
&=&\dsum\limits_{n=0}^{\infty }g_{n+\ell }\left[ x_{1},\ldots
,x_{k};z^{k-1+n}\right] =\left[ x_{1},\ldots ,x_{k};z^{k-1}g_{\ell }\left(
z\right) \right]
\end{eqnarray*}%
which completes the proof.
\end{proof}

\begin{proof}[Proof of Theorem~\protect\ref{th-main-theorem-2}]
By Theorem~\ref{th-main-theorem-1}, we have 
\begin{equation*}
G_{k,\ell }\left( x_{1},\ldots ,x_{k}\right) =\left[ x_{1},\ldots
,x_{k};z^{k-1}g_{\ell }\left( z\right) \right]
\end{equation*}%
and Theorem~\ref{th-main-theorem-2} is a consequence of Lemma~\ref%
{lemma-divdiff-representation}.
\end{proof}

\strut

\thispagestyle{empty}

~\vfill

\end{document}